\documentclass[12pt]{article}
\usepackage{amsmath,amsthm,amsfonts,amssymb,amscd, amsxtra, mathrsfs}
\usepackage{url}
\usepackage[margin=2.7 cm,nohead]{geometry}
\usepackage{color}
\usepackage{verbatim}
\usepackage{url}
\usepackage{cite}
\usepackage{enumerate}
\newtheorem{theorem}{Theorem}
\newtheorem{lemma}[theorem]{Lemma}

\newtheorem{proposition}[theorem]{Proposition}

\newtheorem{remark}{Remark}

\DeclareMathOperator{\C}{\mathcal C}

\DeclareMathOperator{\LL}{LL}

\newcommand{\lng}{\langle}
\newcommand{\rng}{\rangle}
\newcommand{\lf}{\left}
\newcommand{\rg}{\right}
\newcommand{\R}{\mathbb R}
\newcommand{\N}{\mathbb N}

\newcommand{\Hy}{\mathcal H}

\newcommand{\bs}{\left(\begin{smallmatrix}}
\newcommand{\es}{\end{smallmatrix}\right)}

\newcommand{\NI}{\noindent}
\newcommand{\mn}{\mathbb{R}^n_{\ge+}}

\begin{document}

\title{The monotone extended second order cone and mixed complementarity problems}

\author{Y. Gao, S. Z. N\'emeth and R. Sznajder}

\maketitle

\begin{abstract}
In this paper, we study a new generalization of the Lorentz cone ${\cal L}^n_+$, called the Monotone Extended Second Order Cone (MESOC). We
investigate basic properties of MESOC including computation of its Lyapunov rank and proving its reducibility. Moreover, we show that in an
ambient space, a cylinder is an isotonic projection set with respect to MESOC. We also examine a Nonlinear Complementarity Problem on a cylinder, which is equivalent to a suitable Mixed Complementarity Problem and provide a computational example illustrating applicability of MESOC.
\end{abstract}

{\bf Keywords:} {Monotone extended second order cone, Lyapunov rank, Complementarity problems}

\section{Introduction}
In recent years, the second order cone ${\cal L}^n_+:= \{(x_0,x^{n-1}) \in \R \times \R^{n-1}: x_0 \geq \|x^{n-1}\|\}$, also known as the Lorentz cone, attracted much attention of the researchers in optimization, particularly in conic optimization. Many optimization problems can be reformulated as the conic ones. There are computationally stable numerical algorithms for solving various such problems, including complementarity problems. The literature on the subject is vast and commonly accessible, see e.g., Alizadeh and Goldfarb \cite{alizadeh-goldfarb}, and a recent work \cite{hao-wan-chi-jin}, by Hao et al., which is related to the Mixed Complementarity Problem over the second order cone.

The Lorentz cone has a particularly regular structure: it is a {\em self-dual} cone, whose base is isometric to the Euclidean unit ball in $\R^{n-1}$ and is {\em irreducible}. In the context of Euclidean Jordan algebras, ${\cal L}^n_+$ is a symmetric cone (of squares) in the spin algebra ${\cal L}^n$ it generates. We will not pursue this direction here.

There are several known important versions of the extended Lorentz cone, including {\em Bishop--Phelps cone} \cite{gowda-trott} and the Extended Second Order cone \linebreak (ESOC), which was recently developed by S. Z. N\'emeth and his co-authors, see \cite{nemeth-zhang1,nemeth-zhang2,nemeth-xiao1,nemeth-xiao2,nemeth-xie-zhang}. The {\em Lyapunov rank} of a cone $K$, denoted by $\beta(K)$ (see its definition in the next section) is an invariant which shows that the Lorentz cone and ESOC are generally not linearly isomorphic. It was introduced and studied by F. Alizadeh et al. in \cite{alizadeh-etal} under the name of {\em bilinearity rank}. The Lyapunov rank of ${\cal L}^n_+$ was computed in \cite{alizadeh-etal} and \cite{gowda-tao}, and equals $\frac{n^2-n}{2}$. Orlitzky, in \cite{orlitzky}, showed that the latter quantity is the maximum value the Lyapunov rank can attain for a proper cone in $\R^n$. Sznajder, in \cite{sznajder}, showed that the ESOC is irreducible and computed its Lyapunov rank, which is generally lower than $\frac{n^2-n}{2}$.

In this article, we study another extension of the Lorentz cone, called the {\em Monotone Extended Second Order Cone} (MESOC)
\cite{ferreira-gao-nemeth}. There are three main results related to MESOC
here:

\vspace{-.3cm}
\begin{itemize}
  \item[$\circ$]\, computing its Lyapunov rank, which turns out, in general, is much lower than the minimal upper bound indicated in \cite{orlitzky},
  \item[\mbox{}]\, \vspace{-1.1cm}
  \item[$\circ$]\, proving that MESOC (in contrast to ESOC) is a {\em reducible} cone,
  \item[\mbox{}] \vspace{-1.1cm}
  \item[$\circ$]\, showing that a closed convex set is an {\em isotonic
	  projection set} with respect to MESOC if and only if it is a
	  {\em cylinder} (in an ambient space).
\end{itemize}

In \cite{ferreira-gao-nemeth} an application of MESOC to Portfolio Optimization
has been presented and possible other applications have been suggested.

The paper is organized as follows: In Section 2, we collect the necessary definitions and provide examples of monotone cones. The main concept related to a cone $K$, on which the paper relies upon, is the {\em complementarity set of $K$}. In Section 3, for MESOC, we identify its dual space and investigate the structure of its complementarity set. We also formulate here and prove the results listed above. In Section 4, based on the work done in \cite{nemeth-zhang1} and \cite{nemeth-zhang2}, we study the properties of the Mixed Complementarity Problem (MiCP). By exploring the relationship of mixed complementarity problem and nonlinear complementarity problem derived in \cite{nemeth-zhang1}, and by using the isotonicity of MESOC obtained in Section 3, we generate a fixed point iteration sequence (called Picard iteration by some authors), which is convergent to a solution of the MiCP on a general closed and convex cone. The convergence of this iteration is order-based, rather than based on a usual contraction mapping principle, although the preprint \cite{banach-lorentz} and the example in the final section suggests that in certain situations it may be implicitly related to such a principle. This example is about a real MiCP example. We show the existence of a solution, in exact numbers, by using the above iteration.

\section{Preliminaries}

Denote the canonical unit vectors of $\R^n$ by $e^1,\ldots,e^n$ and let $e=e^1+\cdots +e^n$.
Any vector $z\in\mathbb R^n$ is considered to be a column vector and can be uniquely written as
$z=(z_1,\ldots, z_n)^\top:=z_1e^1+\cdots+z_ne^n$. In particular $e=(1,\ldots,1)^\top$.

The canonical {\em inner product} of any two vectors $x,y\in\R^n$ is defined as \[\langle
x,y\rangle:=x^\top y=x_1y_1+\cdots+x_ny_n.\] We identify $\R^p\times\R^q$ with $\R^{p+q}$ through $(x,y)=(x^\top,y^\top)^\top$.

We call the set
\begin{equation*}
	\Hy(u,\alpha):=\{x\in \R^n:\;\lng x,u\rng=\alpha\}
\end{equation*}
an \emph{affine hyperplane}
with the normal $u\in\R^n\setminus\{0\}$ and the corresponding sets
\[\Hy_-(u,\alpha):=\{x\in \R^n:\lng x,u\rng\le\alpha\},\]
\[\Hy_+(u,\alpha):=\{x\in \R^n:\lng x,u\rng\ge\alpha\},\]
\emph{closed half-spaces}.
An \emph{affine hyperplane} through the origin will be simply called \emph{hyperplane}.

A nonempty set $K\subseteq\R^n$ is a {\em cone} if for any $x\in K$ and $\forall \alpha > 0$, it holds $\alpha x \in K$. A set $K$ is a {\em convex cone} (i.e., cone $K$ is a convex set) if and only if for any $x,y\in K$ and $\forall \alpha,\beta > 0$, it holds $\alpha x + \beta y \in K.$

\NI A cone $K$ is called a {\em closed cone (pointed)} when it is a closed set ($K \cap -K =\{0\}$).

\NI The dual cone of a cone $K$ is given by
$$
K^* := \{y \in \R^n :\langle x, y\rangle\geq 0,  \forall x\in K\}.
$$

\NI We define the following set, which is vital for our further considerations
$$
C(K):= \{(x, y): x\in K, y\in K^*,\, x \perp y\},
$$
called the {\em complementarity set of $K$}, where $x\perp y$ means $\langle x,y\rangle = 0$.

A cone $K\subseteq\R^n$ is called \emph{simplicial} if there is a basis $\{u^i:1\le i\le n\}$ of
$\R^n$ such that
\[
K=\lf\{\alpha_1 u^1+\dots+\alpha_n u^n\,:\,\alpha_i\ge 0,\,1\le i\le n\rg\}.
\]
The vectors $u^i$, $1\le i\le n$ are called the \emph{generators} of $K$. It is known that the
dual of a simplicial cone is also simplicial.

We present two examples of complementarity sets, the second will be used later.

\NI {\bf Example 1.}
\, Define the {\em monotone cone} $\R_{\geq}^n$ as
$$
\R_{\geq}^n :=\{x\in\R^n:x_1\geq x_2\geq \cdots\geq x_n\}.
$$
It is easy to check that its dual cone $(\R_{\geq}^n)^*$ is given by
$$
(\R_{\geq}^n)^*=\left \{y\in\R^n: \sum_{i=1}^{j} y_j\geq 0,~ j=1,2,\ldots,n-1,~ \sum_{i=1}^{n}y_i=0 \right \}.
$$
It is an important object, also known as the {\em Schur cone} (see \cite{seeger-sossa}, Example 7.4), since it induces the so-called {\em Schur ordering}, which plays an important role in the theory of majorization, see \cite{marshall-olkin-arnold}.

\NI The complementarity set $C(\R_{\geq}^n)$ of the cone $\R_{\geq}^n$ is described as
\begin{eqnarray*}
C(\R_{\geq}^n)=\Big\{(x,y):x\in \R_{\geq},y\in (\R_{\geq}^n)^*,~ (x_i-x_{i+1})\sum_{j=1}^{i} y_j=0,\\~\forall i=1,2,\ldots,n-1\Big\}.
\end{eqnarray*}

\NI {\bf Example 2.} \label{example2}\,
We define the {\em monotone nonnegative cone} $\R_{\geq+}^n$ as:
$$\R_{\geq+}^n :=\{x\in\R^n: x_1\geq x_2 \geq\cdots\geq x_n\geq 0\}\,.$$
Its dual cone is given by:
$$(\R_{\geq+}^n)^*=\left \{y\in\R^n: \sum_{i=1}^{j}y_i\geq 0,\, j=1,2,\ldots,n\right \},$$
and the complementarity set of\, $\R_{\geq+}^n$ is equal to

\begin{eqnarray*}
C(\R_{\geq+}^n)=\left\{x\in\R_{\geq+}^n,y\in(\R_{\geq+}^n)^*: \Big(x_j=x_{j+1} \textrm{ or } \sum_{i=1}^{j}y_j=0\right., \\ \left.\textrm{ }\forall j=1,2,\ldots,n-1\Big),  \textrm{and}\Big(x_n=0 \textrm{\, or } \sum_{i=1}^{n}y_i=0\Big)\right\}.
\end{eqnarray*}
Both $\mn$ and $(\mn)^*$ are simplicial cones.

\bigskip

Recall \cite{nemeth-zhang1} that the extended second order cone ({\rm ESOC}) is defined by
\[
\textrm{ESOC}=\{(x,u)\in\R^p\times\R^q:\,x\geq\|u\|e\}
\]
and its dual cone is given as
\[
\textrm{(ESOC)}^*=\{(x,u)\in\R^p\times\R^q:\,\langle x,e\rangle\geq\|u\|,\,x\geq 0\},
\]
where $p$ and $q$ are nonnegative integers.

\NI A matrix $A\in\R^{n\times n}$ is called {\em Lyapunov-like} on $K$, if
\begin{equation} \label{lyap-like}
\langle Ax,y\rangle = 0, \text{ }\forall (x,y) \in C(K).
\end{equation}
Define a vector space $\LL(K)$ as the set of all Lyapunov-like matrices on $K$ and denote its dimension as $\beta(K)$, which we call the {\em Lyapunov rank} (or {\em bilinearity rank}) of $K$.

\NI For an arbitrary closed convex set $C \subseteq \mathbb{R}^m$, we define mapping $P_C$--{\em metric projection} onto $C$:
$$\mathbb{R}^m\ni x\mapsto P_Cx:=\arg\min\{\|y-x\| : y\in C\}.$$

\NI Necessarily, $P_C$ is a point--to--point mapping, which is well defined from $\mathbb{R}^m$ onto $C$.
 We also indicate that the projection $P_C$ is {\em nonexpansive}, i.e., for any $x,y\in\mathbb{R}^m$,
\begin{equation}\label{nonexpansivity}
 	\|P_C(x)-P_{C}(y)\|\leq\|x-y\|.
\end{equation}

\NI For any pointed closed convex cone $K\subset\mathbb{R}^m$, a mapping $F: \mathbb{R}^m \rightarrow \mathbb{R}^m$ is called
{\em $K$-isotone} if for any $x, y \in K$, $x\leq_{K} y$ implies $F(x)\leq_{K} F(y)$; here $x\leq_{K} y$ means $y-x \in K$.
If the projection $P_C$ is $K$-isotone, then the closed convex set $C \subseteq \mathbb{R}^m$ is called a {\em $K$-isotone projection set}.

Finally, for a proper closed convex cone $K \subset \mathbb{R}^m$ and a mapping $F: \mathbb{R}^m \rightarrow \mathbb{R}^m$ we define a {\em complementarity problem} CP$(K,F)$ as to find an $x^* \in K$ such that $F(x^*) \in K^*$ and $x^* \perp F(x^*)$. In other words, we seek an $x^*$ such that $(x^*,F(x^*)) \in C(K)$.

\section{Main Results}
The first topic we are interested in is the complementarity problem based on the monotone extended second order cone, which we introduce below.
Let $p$ and $q$ be two nonnegative integers.
{\em The monotone extended second order cone} (informally MESOC) is defined as follows:
 \begin{equation}
L:=\{(x,u)\in\mathbb{R}^p\times\mathbb{R}^q:x_1\geq x_2\geq\cdots\geq x_p\geq\|u\|\}\label{defmesoc}\,.
\end{equation}

\NI In order to find solutions of a complementarity problem, first we need
to find the dual cone and the complementarity set of this cone. Although
a considerable part of the characterization has been already presented in
\cite{ferreira-gao-nemeth}, for the sake of completeness, we decide to
include it here.
\begin{proposition}\label{pd}
Let $p$ and $q$ be two nonnegative integers.  Then the dual cone of a monotone
extended second order cone $L$ in \eqref{defmesoc} is
\begin{equation}\label{dualmesoc}
	{\small M:=\left
	\{(x,u)\in\mathbb{R}^p\times\mathbb{R}^q:\sum_{i=1}^{j}x_i\geq 0, \forall
	j\in\{1,\ldots,p-1\}, \sum_{i=1}^{p}x_i\geq\|u\| \right  \}},
\end{equation}
that is, $M=L^*$.
\end{proposition}
\begin{proof}\, First, we show that $M\subseteq L^*$. Let $(x,u)\in L$ and $(y,v)\in M$. Using Abel's summation formula, we have
\begin{eqnarray*}
\langle (x,u),(y,v)\rangle=
x^Ty+u^Tv= \sum_{i=1}^{p-1}(x_i-x_{i+1})\sum_{j=1}^{i}y_j+x_p\sum_{i=1}^{p}y_i+u^Tv\\\geq\|u\|\|v\|+u^Tv\geq 0.
\end{eqnarray*}
So, we have $M\subseteq L^*$. Now, we show that $L^*\subseteq M$. For, let $(y,v)\in L^*$ and $e=(1,1,\ldots,1) \in \R^p$. It is obvious that $(\|v\|e,-v) \in L$.
Suppose $v\neq 0$, then
$$ \langle (\|v\|e,-v), (y,v)\rangle \geq 0\Leftrightarrow \|v\|\sum_{i=1}^{p}y_i-\|v\|^2\geq 0.$$
Hence, $\displaystyle \sum_{i=1}^{p}y_i\geq \|v\|.$ When $v=0$, then\, $(e,0)\in L$ and $ (y,0)\in L^*$ imply that
$\displaystyle \sum_{i=1}^{p}y_i\geq 0=\|v\|.$

\NI We also have $\Big ((\underbrace{1,1,\ldots,1,0}_{i<p},\underbrace{0,\ldots,0}_{p-i}),(\underbrace{0,0,\ldots,0}_q)\Big)\in L$, and $(y,v)\in L^*$, which implies that
$$
\sum_{j=1}^{i}y_i\geq 0, \text{ }\forall i\in\{1,2,\ldots,p-1\}.
$$
Thus, $(y,v)\in M$, so $L^*\subseteq M$. Altogether, we have $L^* = M$. \hfill $\Box$
\end{proof}
\NI After finding the dual of the monotone extended second order cone, we will describe the complementarity set of this cone. In order to do so, we need to use the inequality, introduced in Lemma \ref{p3} below.

\begin{lemma}\label{p3}
For every $(x,u) \in L$ and $(y,v) \in M$, we have
\begin{equation*}
	\langle x,y \rangle\geq\|u\|\sum_{i=1}^{p}y_i\geq\|u\|\|v\|\,.
\end{equation*}
\end{lemma}

\begin{proof}\, First, we prove that $\langle x,y\rangle\geq\|u\|\sum_{i=1}^{p}y_i$.
Since $(x,u)\in L, (y,v)\in M$, it follows that
$
x_1\geq x_2\geq\cdots\geq x_p\geq\|u\|,
$
$
\sum_{i=1}^{j}y_j\geq 0$ for all $j\in\{1,\ldots,p-1\}$ and  $\sum_{i=1}^{p}y_i\geq\|v\|\geq 0.
$
Thus, by using the backward induction,
\begin{equation*}
\begin{aligned}
\sum_{i=1}^{p}y_i &=y_1+y_2+\ldots+y_p \geq 0\\
\implies &(x_p-\|u\|)\sum_{i=1}^{p-1}y_i + (x_p-\|u\|)y_p\geq 0\\
\implies &(x_{p-1}-\|u\|)\sum_{i=1}^{p-2}y_i + (x_{p-1}-\|u\|)y_{p-1}+(x_p-\|u\|)y_p\geq 0\\
\implies& (x_{p-2}-\|u\|)\sum_{i=1}^{p-3}y_i +(x_{p-2}-\|u\|)y_{p-2}\\& + (x_{p-1}-\|u\|)y_{p-1}+(x_p-\|u\|)y_p\geq 0\\
\cdots&\\
\implies& (x_1-\|u\|)y_1+(x_2-\|u\|)y_2+\cdots+(x_p-\|u\|)y_p\geq 0\\
\iff& \langle x,y\rangle\geq\|u\|\sum_{i=1}^{p}y_i\,.
\end{aligned}
\end{equation*}
Finally, since $\langle x,y\rangle\geq\|u\|\sum_{i=1}^{p}y_i$ and $\sum_{i=1}^{p}y_i\geq\|v\|$, we have
$$
\langle x,y \rangle\geq\|u\|\sum_{i=1}^{p}y_i\geq\|u\|\|v\|\,.
$$
\end{proof}
\vskip -1.1 cm \hfill $\Box$

\NI By using Lemma \ref{p3}, we find the complementarity set of the monotone extended second order cone.

\begin{proposition}\label{cl}
Let $(x,y,u,v)\in C(L)$\footnote{By a slight abuse of the notation, we write $(x,u,y,v)$ instead of $((x,u),(y,v))$.}.
If $u\neq 0, v\neq 0$, then
\begin{gather*}
C(L)=\Bigg\{(x,u,y,v): (x,u)\in L,\textrm{ }(y,v)\in M,\textrm{ }\\\langle x,y\rangle=\|u\|\sum_{i=1}^{p}y_i,\textrm{ }~
 \sum_{i=1}^{p}y_i=\|v\|,\textrm{\rm and }\exists\lambda> 0 ~\textrm{\rm such that } v=-\lambda u\Bigg \}\\=\Bigg\{(x,u,y,v): (x,u)\in L, \textrm{ }(y,v)\in M,
\textrm{ }(x_i-x_{i+1})\sum_{j=1}^{i}y_j=0,\textrm{ }\\\forall i=1,\ldots,p-1, \textrm{ } x_p=\|u\|,\textrm{ }\sum_{i=1}^{p}y_i=\|v\|, \textrm{\rm and }\exists\lambda>0 ~\textrm{\rm such that }  v=-\lambda u\Bigg\}.
\end{gather*}
\end{proposition}

\begin{proof}\, Let
\begin{gather*}
S:=\Bigg\{(x,u,y,v): (x,u)\in L,\textrm{ }(y,v)\in M,\textrm{ }\\\langle x,y\rangle=\|u\|\sum_{i=1}^{p}y_i,\textrm{ }~
 \sum_{i=1}^{p}y_i=\|v\|,\textrm{\rm and }\exists\lambda> 0 ~\textrm{\rm such that } v=-\lambda u\Bigg \}
\end{gather*}

\NI Now, our task is to show that $C(L)=S$. First, we need to prove that $C(L)\subseteq S$.
For arbitrary $(x,u,y,v)\in C(L)$, by using Lemma \ref{p3}, we have
\begin{equation*}
\begin{aligned}
0 = \langle(x,u),(y,v)\rangle & = \langle x,y\rangle+\langle u,v\rangle\\
                              & \geq \|u\|\sum_{i=1}^{p}y_i + \langle u,v\rangle\\
                              & \geq \|u\|\|v\| + \langle u,v\rangle \geq 0\,.
\end{aligned}
\end{equation*}
Hence, all the inequalities above must be equalities, that is,
\begin{equation*}
\begin{aligned}
0=\langle x,y\rangle+\langle u,v\rangle & = \|u\|\sum_{i=1}^{p}y_i + \langle u,v\rangle \\
& = \|u\|\|v\| + \langle u,v\rangle = 0\,.
\end{aligned}
\end{equation*}
Thus,
\begin{equation}
\langle x,y\rangle=\|u\|\sum_{i=1}^{p}y_i=\|u\|\|v\|\label{e1}\,.
\end{equation}
Therefore,
\begin{equation*}
\|u\|\sum_{i=1}^{p}y_i=\|u\|\|v\|
\end{equation*}
and
\begin{equation} \label{e3}
\|u\|\|v\|+\langle u,v\rangle=0\,.
\end{equation}

\NI From \eqref{e1} we get $\langle x,y\rangle=\|u\|\sum_{i=1}^{p}y_i$ and, subsequently, $\sum_{i=1}^{p}y_i=\|v\|$. From the equality case in the Cauchy-Schwarz inequality, equation (\ref{e3}) implies that $\exists\lambda>0, v=-\lambda u$. Thus, $C(L)\subseteq S$.

\NI Now, for the converse inclusion $S\subseteq C(L)$. We have: $\forall (x,u,y,v)\in S$, $\exists\lambda>0$ such that $v = -\lambda u $, $(x,u)\in L, (y,v)\in M, x^Ty=\|u\|\sum_{i=1}^{p}y_i,\forall i=1,\ldots,p$, and $\sum_{i=1}^{p}y_i=\|v\|$. Thus
$$\langle(x,u),(y,v)\rangle = \langle x,y\rangle +\langle u,v\rangle = \|u\|\|v\| + \langle u,v\rangle = 0\,.$$
Therefore, $(x,u,y,v)\in C(L)$. Hence, $S\subseteq C(L)$.

\NI Finally, we have
\begin{gather}\label{cl1}
C(L)=S=\Bigg\{(x,u,y,v): (x,u)\in L,\textrm{ }(y,v)\in M,\textrm{ }\notag\\\langle x,y\rangle=\|u\|\sum_{i=1}^{p}y_i,\textrm{ }~
 \sum_{i=1}^{p}y_i=\|v\|,\textrm{\rm and }\exists\lambda> 0 ~\textrm{\rm such that } v=-\lambda u\Bigg \}
\end{gather}
\vskip -.3cm
\NI Moreover,
$$
\begin{aligned}
\|u\|\sum_{i=1}^{p}y_i & = \langle x,y\rangle\\
& = y_1(x_1-x_2)+(y_1+y_2)(x_2-x_3)+\cdots\\&+(y_1+y_2+\cdots+y_{p-1})(x_{p-1}-x_p)+(y_1+y_2+\cdots+y_p)x_p \\
\end{aligned}
$$
if and only if

\begin{eqnarray*}
(\|u\|-x_p)\sum_{i=1}^{p}y_i = y_1(x_1-x_2)+(y_1+y_2)(x_2-x_3)\\+\cdots+(y_1+y_2+\cdots+y_{p-1})(x_{p-1}-x_p)\,.
\end{eqnarray*}

In the equation above, it is obvious that the LHS (left-hand side) is nonpositive and the
RHS (right-hand side) is nonnegative, thus both must be equal to 0. Since the components of the sum in the RHS are all nonnegative, each component must be equal to $0$. Hence, from equation \eqref{cl1} it follows that

\begin{gather*}
C(L)=\Bigg\{(x,u,y,v): (x,u)\in L, \textrm{ }(y,v)\in M,
\textrm{ }(x_i-x_{i+1})\sum_{j=1}^{i}y_j=0,\textrm{ }\\\forall i=1,\ldots,p-1, \textrm{ } x_p=\|u\|,\textrm{ }\sum_{i=1}^{p}y_i=\|v\|, \textrm{\rm and }\exists\lambda>0 ~\textrm{\rm such that }  v=-\lambda u\Bigg\}.
\end{gather*}
Now the proof is complete.
\end{proof}
\vskip -1.1cm \hfill $\Box$

\begin{lemma}\label{lyap_rank_lem}
	Let $A\in\R^{p\times p}$. Then, $A\in\LL(\R^p_{\ge+})$ if and only if it is of the form
	\begin{equation}\label{llmn}
		A=\left[
		\begin{array}{cccccc}
			a-\sum_{i=2}^pa_i & a_{2} & a_{3} & \cdots & \cdots & a_{p}\\
			& a-\sum_{i=3}^pa_i & a_{3} & \cdots & \cdots & a_{p}\\
			&  & a-\sum_{i=4}^pa_i & \cdots & \cdots & a_{p}\\
			&  &  &  \ddots& \vdots & \vdots\\
			& \emph{\textbf{0}} &  &  & a-a_p & a_{p}\\
			& & & & & a
		\end{array}
		\right],
	\end{equation}
	where $a,a_2,a_3\ldots,a_p\in\mathbb R$ are arbitrary.
\end{lemma}
\begin{proof}
	Let $e^i\in\R^p$, $1\le i\le p$ be the canonical unit vectors in $\R^p$ and
	$e^{p+1}$ be the zero vector in $\R^p$. Denote $u^i:=\sum_{k=1}^i e^k\in\mn$\, and\,
	$v^i:=e^i-e^{i+1}\in(\mn)^*$, for\, $1\le i\le p$ (see Example \ref{example2}). Then,
	$\lng u^i,v^j\rng=\delta_{ij}$, where
	$\delta_{ij}$ is the Kronecker symbol, that is, $\delta_{ii}=1$ and $\delta_{ij}=0$, for
	$i\ne j$. If follows that $(u^i,v^j)\in C(\mn)$, whenever $i\ne j$ (as it can be seen
	from Example \ref{example2}, too). Hence, if
	$A\in\LL(\mn)$ and $i\ne j$, then
	\begin{equation}\label{ij}
		\lng Au^i,v^j\rng=\sum_{k=1}^i (a_{jk}-a_{j+1,k})=0,
	\end{equation}
	where we set $a_{p+1,k}:=0$. By using equation \eqref{ij}, we get  	
	\begin{equation}\label{ij2}
		\sum_{\ell=j}^p\lng Au^i,v^{\ell}\rng=\sum_{k=1}^i a_{jk}=0,\quad \mbox{if}\quad
		j>i.
	\end{equation}
	By equation \eqref{ij2} we get
	\begin{equation}\label{ij3}
		a_{ji}=\sum_{k=1}^i a_{jk}-\sum_{k=1}^{i-1} a_{jk}=0, \quad \mbox{if}\quad
		j>i.
	\end{equation}
	By using again equation \eqref{ij}, we get
	\begin{equation}\label{ij4}
		a_{ji}-a_{j+1,i}=\sum_{k=1}^i (a_{jk}-a_{j+1,k})
		-\sum_{k=1}^{i-1} (a_{jk}-a_{j+1,k})=0,\quad \mbox{if}\quad j+1<i.
	\end{equation}
	Equations \eqref{ij}, \eqref{ij3} and \eqref{ij4} imply that $A$ is of the form \eqref{llmn}. Now,
	suppose that $A$ is of the form \eqref{llmn}. From Example \ref{example2}, any element
	$(x,y)\in C(\mn)$ can be written in the form
	\begin{equation}\label{xy}
		(x,y)=\lf(\sum_{i\in I}\alpha_i u^i,\sum_{j\in J}\beta_j v^j\rg);\quad
		\alpha_i,\beta_j\ge 0,
	\end{equation}
	for some $I,J\subseteq\{1,2,\dots,n\}$ with $I\cup J=\{1,2,\dots,n\}$ and
	$I\cap J=\varnothing$, because $\{u^i:1\leq i\leq j\}\subseteq\mn$ and
	$\{v^i:1\leq i\leq j\}\subseteq(\mn)^*$ are generators of the simplicial cones
	$\mn$ and $(\mn)^*$, respectively, and $x\perp y$. As
	\(\lng Au^i,v^j\rng=0,\) by considering the derivation of equations \eqref{ij},
	\eqref{ij2}, \eqref{ij3} and \eqref{ij4} above in the reverse order, equation \eqref{xy}
	implies that $\lng Ax,y\rng=0$. Hence, $A\in\LL(\mn)$.
\end{proof}
\begin{theorem}
For the monotone extended second
order cone \eqref{defmesoc}, any Lyapunov like transformation $T$ is of the form
{\small
\begin{equation}\label{ee4}
T=\left[
\begin{array}{cccccc|ccc}
a-\sum_{j=2}^pa_{j} & a_{2} & a_{3} & \cdots & \cdots & a_{p} & c_1 & \cdots & c_q\\
 & a-\sum_{j=3}^pa_{j} & a_{3} & \cdots & \cdots & a_{p} & c_1 & \cdots & c_q \\
 &  & a-\sum_{j=4}^pa_{j} & \cdots & \cdots & a_{p} & c_1 & \cdots & c_q \\
 &  &  &  \ddots& \vdots & \vdots & \vdots & & \vdots\\
 & \emph{\textbf{0}} &  &  & a-a_{p} & a_{p} & c_1 & \cdots & c_q \\
 & & & & & a & c_1 & \cdots & c_q
\\ \hline
 &  &  & &  & c_1 & a & & *\\
 & \emph{\textbf{0}} &  & &  & \vdots & & \ddots & \\
 &  &  & &  & c_q & -* & & a\\
\end{array}
\right],
\end{equation}
}
where $a,a_2,a_3,\ldots,a_p,c_1,...,c_q\in\R$ are arbitrary.
Hence, its
Lyapunov rank is given by
\begin{equation*}
\beta(L)=p+\frac{q(q+1)}2\,.
\end{equation*}
\end{theorem}

\begin{proof}\, Recall that the complementarity set for the monotone extended second order cone $L$ is
$$C(L)=\{((x,u),(y,v)) \in L\times M: (x,u)\perp (y,v)\}.$$
We partition the above set in the following way:
$$C(L):=C_1(L)\cup C_2(L)\cup C_3(L)\cup C_4(L),$$
where
\begin{equation*}
\begin{aligned}
	& C_1(L):=\{(x,0,y,0)\in C(L)\},\\
	& C_2(L):=\{(x,0,y,v)\in C(L):v\neq 0\},\\
	& C_3(L):=\{(x,u,y,v)\in C(L):u\neq 0\neq v\},\\
	& C_4(L):=\{(x,u,y,0)\in C(L):u\neq 0\}.
\end{aligned}
\end{equation*}

\NI Since $x=0 \Rightarrow u=0$ and $y=0 \Rightarrow v=0$, for any Lyapunov-like transformation on
$L$ we only need to consider the case of\, $x\neq 0 \neq y$.  Let $T$ be any element of $\LL(L)$, so it has the following block form:

$$
\left[
\begin{array}{cc}
   A & B  \\
   C & D
  \end{array}
  \right]
  : \R^p\times\R^q\rightarrow\R^p\times\R^q,
$$
where $A\in\R^{p\times p},B\in\R^{p\times q}, C\in\R^{q\times p}$, and $D\in\R^{q\times q}$. Take any $(x,u,y,v) \in C(L)$. Then (\ref{lyap-like}) implies
$$
\langle Ax,y\rangle + \langle Bu,y\rangle + \langle Cx,v\rangle + \langle Du,v\rangle =0,
$$
$$
\langle Ax,y\rangle - \langle Bu,y\rangle - \langle Cx,v\rangle + \langle Du,v\rangle =0,
$$
where the latter equation comes from the former one by substituting $-u$ for $u$ and $-v$ for $v$. By adding and subtracting the above equations, we get
\begin{equation}
\begin{aligned}
\langle Ax,y\rangle + \langle Du,v\rangle=0,\\
\langle Bu,y\rangle + \langle Cx,v\rangle=0.\label{ee}
\end{aligned}
\end{equation}

\NI By using an element $(x,0,y,0)\in L\times M$ in $C_1(L)$, with $x\in\R_{\geq+}^p$ and
$y\in(\R_{\geq+}^p)^*$, we get $\langle Ax,y\rangle = 0$, which implies that $A\in\LL(\R_{\geq+}^p)$.

Now, we will determine the structures of matrices $B$ and $C$. By using elements in $C_2(L)$, from the second equation in (\ref{ee}), we get
\begin{equation*}
	\langle Cx,v\rangle = \langle B0,y\rangle + \langle Cx,v\rangle = 0.
\end{equation*}
Suppose that $Ca^i\neq 0$ for some $i<p$ and let $v:=\frac{Ca^i}{\|Ca^i\|}$, and $y:=e^j,(j>i)$, thus, $\langle y, e^j\rangle = 1 = \|v\|$. Hence, $(a^i,0,e^j,v)\in C_2(L)$. Then $0=\langle Cx,v\rangle=\langle Ca^i,v\rangle =\|Ca^i\|$, which leads to a contradiction. Hence, $Ca^i = 0$.
Then, for certain $c_1,\ldots,c_q\in\R$ we have
$$
C=\left[\begin{array}{c|c}\textbf{0}&\begin{array}{ccc}c_1\\ \vdots \\c_q\end{array}\end{array}\right]_{q\times p}.
$$
If $C=0$, the second equation in (\ref{ee}) demonstrates that $\langle Bu,y \rangle =0$ for all $(x,u,y,v)\in C_3(L)$. It is easy to verify that $(e,-v,e^i,v)\in C_3(L)$, where $v$ is an arbitrary unit vector in $\R^q$. Hence, $\langle B(-v),e^i\rangle = 0$, for all $1 \leq i \leq p$, thus $Bv=0$. In consequence, $B=0$.

If $C \neq 0$, first we need to find the structure of matrix $B$. We have $\langle Bu,y\rangle=0$ for any $(x,u,y,0)\in C_4(L)$.

Let $u^i$ denote the standard (canonical) unit vector in $\R^q$ and for any $n>m$, let
$y^{m,n}:=e^m-e^n \in\R^p$. Since $\left(e,u^i,y^{m,n},0\right)\in C_4(L)$,
$$
\langle Bu^i,y^{m,n}\rangle=0.
$$
Therefore,
$$
B=\begin{bmatrix}
b_1 & b_2 & \cdots &b_q \\
\vdots & \vdots &  & \vdots\\
 b_1 & b_2 & \cdots &b_q
\end{bmatrix}_{p\times q}\,.
$$
For $i=1,\ldots,q$ and $j=1,\ldots,p$, we have $(e,u^i,e^j,-u^i) \in C_3(L)$ and subsequently,
$$\langle Bu^i,e^j\rangle+\langle Ce,-u^i\rangle=0.$$
It readily implies $b_i=c_i$. Hence,
$$
B=\begin{bmatrix}
c_1 & c_2 & \cdots &c_q \\
\vdots & \vdots &  & \vdots\\
 c_1 & c_2 & \cdots &c_q
\end{bmatrix}_{p\times q}\,.
$$

As $(e,u,\frac{1}{p}e, -u) \in C_3(L)$ for all $u$ with $\|u\|=1$, by using (\ref{ee}), we have
\begin{equation} \label{ee2}
\left\langle Ae,\frac{1}{p}e\right\rangle+\langle Du,-u\rangle =0.
\end{equation}

Let $\displaystyle a:=\frac{\langle Ae,e\rangle}{p}$. Then (\ref{ee2}) implies
$$\left\langle\left(\frac{D+D^T}{2}-aI\right)u,u\right\rangle=0,
$$
and hence
\begin{equation}
D+D^T=2aI.\label{ee3}
\end{equation}

\NI Obviously, $(e,-u^1,e^1,u^1) \in C(L)$ and using the first equation in (\ref{ee}) gives
$$
\langle Ae,e^1\rangle -\langle Du^1,u^1\rangle=0,
$$
which implies that $d_{11}=\sum_{j}a_{1j}$. Thus, (\ref{ee3}) implies that $d_{11}=a$ and hence, $\sum_{j=1}^pa_{1j}=a$.

\NI By changing $e^1$ to $e^2$ (yes, we can), we have $\sum_{j=2}^pa_{2j}=d_{22}=a$. By following this process, we obtain that $d_{ii}=\sum_{j=1}^pa_{ij}=a$, for all $1\leq i \leq p$.

Therefore, by equation (\ref{ee3}), $A\in\LL(\mn)$ (shown above) and Lemma \ref{lyap_rank_lem},
any Lyapunov-like transformation on $L$ has the form \eqref{ee4}.

Now, we want to show that any transformation $T$, which can be represented in the form (\ref{ee4}), is Lyapunov-like on $L$, so let $T$ be given as above. Then we have
\begin{equation}\label{ellt}
\langle T(x,u),(y,v)\rangle = \langle Ax,y\rangle+\langle Du,v\rangle+\langle Bu,y\rangle + \langle Cx,v\rangle.
\end{equation}
We wish to show that for any $(x,u,y,v) \in C(L)$, the RHS in the above equation is zero.
We will perform a case-by-case analysis.

\NI {\em Case 1.}\,
For any $(x,u,y,v):=(x,0,y,0)\in C_1(L)$, the RHS of \eqref{ellt} is equal to zero, as
$(x,y)\in\C\left(\R^n_{\ge+}\right)$ and we have already shown that $A\in\LL(\R^n_{\ge+})$, hence
it is enough to use again Lemma \ref{lyap_rank_lem}.

\NI {\em Case 2.}\, For any
$(x,u,y,v):=(x,0,y,v)\in C_2(L)$, the RHS of \eqref{ellt} is $(c_1v_1+...+c_qv_q)x_p$. Suppose
that $x_p\ne 0$. Then, since $(x,y)\in C\lf(\R^n_{\ge+}\rg)$, from Example \ref{example2} we
get
$y_1+\dots+y_p=0$. Hence, $(y,v)\in M$ and \eqref{dualmesoc} implies $v=0$, which contradicts
$(x,0,y,v)\in C_2(L)$. Thus, $x_p=0$ and therefore the RHS of \eqref{ellt} is zero.

\NI {\em Case 3.}\,
Take an arbitrary
$(x,u,y,v) \in C_3(L)$. Proposition \ref{cl} indicates that for some $\lambda >0$ one has $v=-\lambda u$, thus
\begin{equation}\label{ly}
\begin{aligned}
\langle Ax,y\rangle +\langle Du,v\rangle & = \langle Ax,y\rangle +\left\langle \frac{D+D^T}{2} u,v\right\rangle=\langle Ax,y\rangle + a\langle u,v\rangle\\
&=\langle z,y\rangle+a\langle u,v\rangle\\&=\sum_{i=1}^{p-1}\left[(z_i-z_{i+1})\sum_{j=1}^iy^j\right]+z_p\sum_{i=1}^{p}y^i+a\langle u,v\rangle,
\end{aligned}
\end{equation}
where $z_i:=\sum_{j=1}^ia_jx_i+\sum_{k=i+1}^pa_kx_k$, for any $1\le i\le p-1$ and $z_p=\sum_{k=1}^pa_kx_p$. Then for any $1\le i\le p-1$, it is easy check that $z_i-z_{i+1}=\sum_{j=1}^ia_j(x_i-x_{i+1})$. By inserting these equalities and the formula for $z_p$ into equation \eqref{ly}, and by using Proposition \ref{cl}, we obtain $\langle Ax,y\rangle +\langle Du,v\rangle =0$.
We will show that
$$
\langle Bu,y\rangle + \langle Cx,v\rangle=0.
$$
By using the full power of Proposition \ref{cl}, including $v=-\lambda u$ for some $\lambda>0$, we have
$$
\langle Bu,y\rangle = \sum_{i=1}^q(c_iu_i)\cdot\sum_{i=1}^py_i=\|v\|\sum_{i=1}^q(c_iu_i)
$$
and
$$
\langle Cx,v\rangle = x_p\sum_{i=1}^q(c_iv_i)=\|u\|\sum_{i=1}^q(c_iv_i).
$$
Then
\begin{align*}
\langle Bu,y\rangle + \langle Cx,v\rangle &= \|v\|\sum_{i=1}^q(c_iu_i)+ \|u\|\sum_{i=1}^q(c_iv_i)
\\
&=\lambda\|u\|\sum_{i=1}^q(c_iu_i)-\lambda\|u\| \sum_{i=1}^q(c_iu_i)=0.
\end{align*}
\medskip

\NI {\em Case 4.}\,
For any
$(x,y,u,v):=(x,u,0,v)\in C_4(L)$, the RHS of \eqref{ellt} is $(c_1u_1+...+c_qu_q)(y_1+...+y_q)$.
Suppose that $y_1+\dots+y_p\ne 0$. Then, since $(x,y)\in C(\R^n_{\ge+})$, from Example
\ref{example2} we get $x_p=0$.  Hence, $(x,u)\in M$ and \eqref{dualmesoc} implies $u=0$, which
contradicts $(x,u,y,0)\in C_4(L)$. Thus, $y_1+\dots+y_p=0$ and therefore the RHS of
\eqref{ellt} is zero.

In conclusion, the RHS of \eqref{ellt} is zero for any $(x,u,y,v)\in C_1(L)\cup C_2(L)\cup
C_3(L)\cup C_4(L)=C(L)$
Therefore, $T\in\LL(L)$. Following the definition of the Lyapunov rank, its value for the cone $L$ equals to the number of independent parameters in (\ref{ee4}), which is\, $p+\frac{q(q+1)}{2}$. \hfill $\Box$
\end{proof}
After calculating the Lyapunov rank of MESOC, we prove our second main result, namely that this cone is reducible. Recall that a cone $K$ in $\mathbb{R}^m$ is {\em reducible} if it can be expressed as a sum $K=K_1+K_2$, where $K_1, K_2 \neq \{0\}$ are cones with ${\rm span}(K_1) \cap {\rm span}(K_2)= \{0\}$. Otherwise, it is called {\em irreducible}.

\begin{theorem}
For the monotone extended second order cone $L$ defined in (\ref{defmesoc}) one has $L=L_1+L_2$, where
\begin{align*}
L_1:={\rm cone}\,\Big\{ (\underbrace{1,\ldots,1}_p, m_1, \ldots, m_q): m_1^2+ \cdots+m_q^2 \leq 1 \Big\}
\end{align*}
and
\begin{align*}
L_2:={\rm cone}\,\Big \{(\underbrace{1,0,\ldots,0}_p,\underbrace{0, \ldots,0}_q), (\underbrace{1,1,\ldots,0}_p,\underbrace{0, \ldots,0}_q), \\\ldots, (\underbrace{1,1,\ldots,1,0}_p,\underbrace{0, \ldots,0}_q) \Big\}.
\end{align*}
As a by-product, we show that $L$ is a reducible cone.
\end{theorem}

\begin{proof}\, First, we show the inclusion $L\subseteq L_1+L_2$.

\NI An arbitrary element $(x_1, \ldots, x_p, u_1, \ldots, u_q) \in L$, by the definition of $L$, can be represented as $(\sum_{i=1}^p a_i, \ldots, a_1+a_2, a_1, u_1, \ldots, u_q)$, where $a_i \geq 0$ for $i=2,\ldots,p$ and $a_1 \geq \|(u_1, \ldots,u_q)\|$. Hence,
\begin{align*}
& (x_1, \ldots, x_p, u_1, \ldots, u_q)\\
& = \left(\sum_{i=1}^p a_i, \sum_{i=1}^{p-1} a_i, \ldots, a_1, u_1, \ldots, u_q\right) \\
& = (a_1, \ldots, a_1,u_1, \ldots, u_q) + (\underbrace{a_2, \ldots, a_2,0}_p,\underbrace{0, \ldots, 0}_q) + \cdots \\&+ (\underbrace{a_p,0, \ldots,0}_p, \underbrace{0, \ldots, 0}_q)\\
& = (a_1, \ldots, a_1,u_1, \ldots, u_q)+ a_2(\underbrace{1, \ldots, 1,0}_p,\underbrace{0, \ldots, 0}_q)+ \cdots\\ &+
a_p(\underbrace{1,0, \ldots,0}_p, \underbrace{0, \ldots, 0}_q).
\end{align*}
Obviously, $a_2(\underbrace{1, \ldots, 1,0}_p,\underbrace{0, \ldots, 0}_q)+ \cdots + a_p(\underbrace{1,0, \ldots,0}_p, \underbrace{0, \ldots, 0}_q) \in L_2$.

\vspace{.2cm}
\NI Now, we show that $(a_1, \ldots, a_1,u_1, \ldots, u_q) \in L_1$. It is trivial when $a_1=0$, so we assume that $a_1>0$. Thus, we have
\[\displaystyle (a_1, \ldots, a_1,u_1, \ldots, u_q)=a_1\left (1, \ldots, 1,\frac{u_1}{a_1}, \ldots, \frac{u_q}{a_1}\right ).\] As $a_1 \geq \|(u_1, \ldots,u_q)\|$, we get
$$a_1\geq \sqrt{u_1^2+ \cdots + u_p^2}\, \equiv\, 1 \geq \sqrt{\left ( \frac{u_1}{a_1}\right )^2 + \cdots + \left (\frac{u_q}{a_1}\right)^2  }, $$
which, by the definition of $L_1$, gives that $(a_1, \ldots, a_1,u_1, \ldots, u_q) \in L_1$.

\NI Hence, we showed that an arbitrary element $(x_1, \ldots, x_p, u_1, \ldots, u_q) \in L$ can be represented as the sum of two elements, which are
\[(a_1, \ldots, a_1,u_1, \ldots, u_q) \in L_1~\]
and
\[a_2(\underbrace{1, \ldots, 1,0}_p,\underbrace{0, \ldots, 0}_q)+ \cdots + a_p(\underbrace{1,0, \ldots,0}_p, \underbrace{0, \ldots, 0}_q) \in L_2\,.
\]

Now, for the inclusion $L_1+L_2 \subseteq L$. Observe that $L_1 \subseteq L$ and $L_2 \subseteq L$. From the convexity of the cone $L$, it follows that $L_1+l_2 \subseteq L+L=L$.

\vspace{.1cm}
\NI

It concludes the proof of the equality $L=L_1+L_2$. Obviously, the cones $L_1, L_2 \neq \{0\}$ and ${\rm span}(L_1) \cap {\rm span}(L_2)= \{0\}$.     \hfill $\Box$
\end{proof}

For the sake of completeness we quote the following three results that will help us proving Theorem \ref{isotonicity-on-cylinder}, where a characterization of $K\subseteq\R^p\times\R^q$ as an L-isotone projection set is given.

\begin{theorem}[see \cite{nemeth-nemeth-2012}]\label{fooo} 
	The closed convex set $C\subset\R^m$ with nonempty interior is a $K$-isotone projection set
	if and only if it is of the form
	\begin{equation*}
		C=\bigcap_{i\in \N} \Hy_-(u^i,\alpha_i),
	\end{equation*}
	where each affine hyperplane $\Hy(u^i,\alpha_i)$ is tangent to $C$ and it is a $K$-isotone projection set.
\end{theorem}

\NI The following two lemmas are from \cite{nemeth-zhang1}.

\begin{lemma}\label{ti}
	Let $K\subset\R^m$ be a closed convex cone and $\Hy\subset\R^m$ be a hyperplane
	with a unit normal vector $a\in\R^m$. Then, $\Hy$ is a $K$-isotone projection set
	if and only if \[\lng x,y\rng\ge\lng a,x\rng\lng a,y\rng,\] for any $x\in K$ and
	$y\in K^*$.
\end{lemma}

\begin{lemma}\label{leasy}
	Let $z\in\R^m$, $K\subset\R^m$ be a closed convex cone and $C\subset\R^m$ be a nonempty closed convex set. Then, $C$ is a $K$-isotone projection set if and only
	if $C+z$ is a $K$-isotone projection set.
\end{lemma}

Finally, by using the above three results, we derive an isotonicity property of MESOC, which we will use to solve complementarity problems on the MESOC.
\begin{theorem}\label{isotonicity-on-cylinder}
Let $L$ be the MESOC corresponding to the dimensions $p$ and $q$, with  $q>1$. The closed convex set with nonempty interior $K\subseteq\R^p\times\R^q$ is an L-isotone projection set if and only if $K=\R^p\times C$, for some closed convex set with nonempty interior $C\subseteq\R^q$.
\end{theorem}

\begin{proof}\, First, suppose that $K=\R^p\times C$, where $C\subseteq\R^q$ is a nonempty closed convex set with nonempty interior. Let $(x,u),(y,v)\in\R^p\times\R^q$ be such that $(x,u)\leq_{L}(y,v)$, thus $(y-x,v-u)\in L$, i.e.,
\begin{equation}\label{defcon}
y_1-x_1\geq y_2-x_2\geq\cdots\geq y_p-x_p\geq\|v-u\|.
\end{equation}
Since $C$ is a closed and convex set in $\R^q$, by the nonexpansivity (\ref{nonexpansivity}) of
$P_C$, we have
\begin{equation*}
\|v-u\|\geq\|P_Cv-P_Cu\|,
\end{equation*}
which together with (\ref{defcon}) yields
$$
y_1-x_1\geq y_2-x_2\geq\cdots\geq y_p-x_p\geq\|P_Cv-P_Cu\|.
$$
Thus,
$$
(y,P_Cv)-(x,P_Cu)\in L
$$
and therefore we have
$$
P_K(x,u)=(x,P_Cu)\leq_{L}(y,P_Cv)=P_K(y,v).
$$
In conclusion, $K$ is an $L$-isotone project set.

Conversely, suppose that the closed convex set $K\subseteq\R^p\times\R^q$ with nonempty interior  is an $L$-isotone project set. If $p=1$, then in \cite{nemeth-nemeth-2012} it has been proved that $K=\R^p\times C$, where $C$ is a nonempty, closed and convex subset with nonempty interior of $\R^q$. Therefore, assume that $p>1$. By Theorem \ref{fooo} and Lemma \ref{leasy}, we need to show that for any tangent hyperplane $\Hy$ of $K$ with unit normal $\gamma=(a,u)$, we have $a=0$. From Lemma \ref{ti}, we have
\begin{equation}\label{eau}
	\lng\zeta,\xi\rng\ge\lng\gamma,\zeta\rng\lng\gamma,\xi\rng,
\end{equation}
for any $\zeta:=(x,v)\in L$ and $\xi:=(y,w)\in L^*$.
	By Lemma \ref{ti}, condition \eqref{eau} holds.
 	Let $x\in\R^p_+$ and $v\in\R^q$. Then, by equation \eqref{defmesoc}, and Proposition
	\ref{pd}, it is easy to check that $\zeta:=(\|v\|e,v)\in L$, $\xi:=(\|v\|x,-\lng e,x\rng v)\in L^*$ and $\lng\zeta,\xi\rng=0$. Hence,
	condition \eqref{eau} implies
	\begin{equation}\label{ea}
	    0\ge(\lng a,e\rng\|v\|+\lng u,v\rng)(\lng a,x\rng\|v\|-\lng e,x\rng\lng u,v\rng).
	\end{equation}
	If in \eqref{ea} $x=e$ and we choose $v\neq 0$ such that $\lng u,v\rng=0$, then $0\ge\lng a,e\rng^2\|v\|^2$, and hence $\lng a,e\rng=0$.
	Thus, \eqref{ea} becomes
	\begin{equation}\label{ea2}
	    0\ge\lng u,v\rng(\lng a,x\rng\|v\|-\lng e,x\rng\lng u,v\rng).
	\end{equation}
	First, suppose that $u\ne0$. Let $v^n\in\R^q$ be a sequence of points such that $\|v^n\|=1$, $\lng u,v^n\rng>0$ and
	$\lim_{n\to+\infty}\lng u,v^n\rng=0$. Let $n$ be an arbitrary positive integer.
	If in \eqref{ea2} we choose $\lambda>0$ sufficiently large such that $x:=a+\lambda e\ge0$ and $v=v^n$,
	we get \(0\ge\lng u,v^n\rng(\|a\|^2-\lambda p\lng u,v^n\rng),\) or equivalently $\|a\|^2\le\lambda p\lng u,v^n\rng$. By letting
	$n\to+\infty$ in the last inequality, we obtain $\|a\|^2\le0$, or equivalently $a=0$.

    Next, suppose that $u=0$. Since $(a,u)$ is a unit vector, it follows that $a\ne 0$. Let $(x,y)\in C(\R^p_{\ge +})$ and $w\in\R^q$ such that
	$\lng y,e\rng\ge\|w\|$. Then, by \eqref{defmesoc} and Proposition \ref{pd}, it is easy to check that
	$\zeta:=(x,0)\in L$, $\xi:=(y,w)\in L^*$ and $\lng\zeta,\xi\rng=0$. Hence, inequality \eqref{eau} implies
	\begin{equation*}
	    0\ge\lng a,x\rng\lng a,y\rng,
	\end{equation*}
	for any $(x,y)\in C(\R^p_{\ge +})$ with $\lng x,y\rng=0$. From Example \ref{example2}, we can choose
	$x=e^1+\cdots+e^r$ and $y=e^s-e^{s+1}$, where $r, s\in\{1,\ldots,p\}$, and we set $e^{p+1}:=0$. Hence, $(a_1+\cdots +a_r)(a_s-a_{s+1})\le0$, where we set $a_{p+1}:=0$. Take now $r=1$ and for $s=1, \ldots,p$, add the inequalities $a_1(a_1-a_2)\leq 0, \ldots, a_1(a_p-a_{p+1})\leq 0$, to obtain (by the telescoping effect) $a_1\cdot a_1 \leq 0$, which gives $a_1=0$. Similarly, for $r=2$ and $s=2, \ldots,p$, add the inequalities $(0+a_2)(a_2-a_3) \leq 0, \ldots, a_2(a_p-a_{p+1}) \leq 0$, to get $a_2=0$. Acting similarly (with $r=3$, and so on), we get $a_3=0$, up to $a_p=0$. Thus, $a=0$. But this contradicts $a\ne 0$, so the case $u=0$ cannot hold. \hfill $\Box$
\end{proof}	
It is a well-known that for the nonlinear complementarity problem NCP$(F,K)$, $x^*$ is its solution if and only if $x^*$ is a fixed point of the mapping $K\ni x\mapsto P_K(x-F(x))$. For an arbitrary sequence $\{x^n\}$ generated by the fixed point iteration process
\begin{equation}\label{picard iteration}
x^{n+1}=P_{K}(x^n-F(x^n)),
\end{equation}
if the mapping $F$ is continuous and the sequence $\{x^n\}$ is convergent to $x^*\in K$, then $x^*$ is a fixed point of the mapping $K\ni x\mapsto P_K(x-F(x))$, hence $x^*$ is a solution of the nonlinear complementarity problem NCP$(F,K)$.

\section{Mixed Complementarity Problem}
Facchinei and Pang defined the mixed complementarity problem (${\rm MiCP}$) on
the nonnegative orthant (see Subsection 9.4.2 in \cite{facchinei-pang}). It is
not only equivalent to a linearly constrained {\em variational inequality
problem} (this relationship is also known as the {\em Karush-Kuhn-Tucker (KKT)}
system of the variational inequality), but it can also be viewed as an {\rm NCP}
for a particular non-pointed cone. N\'emeth and Zhang \cite{nemeth-zhang1}
considered the MiCP defined on an arbitrary closed and convex cone. In Theorem \ref{isotonicity-on-cylinder}, we have already shown that the projection mapping onto a cylinder is an isotonic projection set with respect to MESOC. It is interesting to consider using isotonicity on MESOC as a tool to solve the MiCP.

\NI For the sake of completeness below we quote Lemma 4 in \cite{nemeth-zhang2}.
\begin{lemma}\label{guohan1}
Let $K=\R^p\times C$, where $C$ is an arbitrary nonempty closed and convex cone in $\R^q$. Denote mapping $G: \R^p\times\R^q\mapsto\R^p$, mapping $H: \R^p\times\R^q\mapsto\R^q$ and mapping $F=(G;H): \R^p\times\R^q\mapsto\R^p\times\R^q$. Then the nonlinear complementarity problem {\rm NCP}$(F,K)$ is equivalent to the mixed complementarity problem {\rm MiCP}$(G,H,C,p,q)$ defined as
$$
G(x,u)=0,\,\,\,C\ni u\perp H(x,u)\in C^*.
$$
\end{lemma}

\begin{proof}\, It is standard and follows from the definition of the nonlinear complementarity problem NCP$(F,K)$, by noting that $K^* = \{0\}\times C^*$.  \hfill $\Box$
\end{proof}
By using the notations of Lemma \ref{guohan1}, the fixed point iteration (\ref{picard iteration}) can be rewritten as:
\begin{equation}\label{picard iteration 2}
\begin{cases}
x^{n+1}=x^n-G(x^n,u^n),  \\
u^{n+1}=P_C(u^n-H(x^n,u^n)).
\end{cases}
\end{equation}
\vspace{-.1cm}

\NI For the sake of self-containment below we quote Proposition 2 in
\cite{nemeth-zhang2}.
\begin{proposition}\label{guohan2}
Let $L\subseteq \R^m$ be a pointed closed convex cone, $K \subseteq \R^m$ be a closed convex cone and $F: \R^m \rightarrow \R^m$ be a continuous mapping. Consider the sequence $\{x^n\}_{n\in\N}$ defined by (\ref{picard iteration}). Suppose that the mappings $P_K$ and $I - F$ are $L-isotone$ and $x^0 = 0 \leq_L x^1$. Let
$$
\Omega:=K\cap L\cap F^{-1}(L)=\{x\in K\cap L: F(x)\in L\}
$$
and
$$
\Gamma:=\{x\in K\cap L:P_K(x-F(x))\leq_L x\}.
$$
Then $\varnothing\ne\Omega\subset\Gamma$ and the sequence $\{x^n\}$ is convergent to $x^*$, which is a solution of {\rm NCP}$(F,K)$. Moreover, $x^*$ is a lower L-bound of $\Omega$ and the L-least element of $\Gamma$.
\end{proposition}

The following theorem provides sufficient conditions for the solvability of
the mixed complementarity problem MiCP$(G,H,C,p,q)$.
\begin{theorem}\label{MainT}
Let $L$ be the monotone extended second order cone corresponding to $p$ and $q$. For an arbitrary cone $K=\R^p\times C$, where $C$ be a closed convex cone, denote its dual cone by $K^*$. Let $F=(G;H): \R^p\times\R^q\mapsto\R^p\times\R^q$, such that $I-F$ is $L$-isotone, where I denotes the identical mapping, $G: \R^p\times\R^q\mapsto\R^p$ and $H: \R^p\times\R^q\mapsto\R^q$ are two continuous mappings. Consider a sequence $\{(x^n,u^n\}_{n\in\N}$ defined by $(\ref{picard iteration 2})$, where $x^0=0\in\R^p$ and $u^0=0\in\R^q$. Let $x,y\in\R^p$ and $u,v\in\R^q$. Suppose that the system of inequalities
\begin{equation}\label{si1}
y_i-x_i\geq y_{i+1}-x_{i+1}\geq\|v-u\|;\quad 1\le i\le p-1
\end{equation}
implies the system of inequalities
\begin{equation}\label{si2}
\begin{array}{rcl}
y_i-x_i-(G(y,v)_i-G(x,u)_i)&\geq& y_{i+1}-x_{i+1}-(G(y,v)_{i+1}-G(x,u)_{i+1})\\
&\geq&\|v-u-(H(y,v)-H(x,u))\|;
\end{array}
\end{equation}
$1\le i\le p-1$, and that $x^1_i\geq x^1_{i+1}\geq \|u^1\|$; $1\le i\le p-1$  (in particular, this holds when $-G(0,0)_i\geq-G(0,0)_{i+1}\geq\|H(0,0)\|$; $1\le i\le p-1$).
Let
$$
\Omega:=\{(x,u)\in\R^p\times C: x_1\geq\cdots\geq x_p\geq\|u\|,G(x,u)_1\geq \cdots \geq G(x,u)_p \geq\|H(x,u)\|\}
$$
and
\begin{eqnarray*}
\Gamma:=\{(x,u)\in\R^p\times C:x_1\geq\cdots\geq x_p\geq\|u\|,G(x,u)_1 \geq\cdots\geq G(x,u)_p\\\geq\|u-P_C(u-H(x,u))\|\}.
\end{eqnarray*}
Then $\varnothing\ne\Omega\subseteq\Gamma$, the sequence $\{(x^n,u^n)\}$ is convergent, and its limit $(x^*,u^*)$ is a solution of {\rm MiCP}$(G,H,C,p,q)$. Moreover, $(x^*,y^*)$ is a lower L-bound of $\Omega$ and the L-least element of $\Gamma$.
\end{theorem}

\begin{proof}\, Following the definition of the monotone extended second order cone, we have
$$
\Omega=K\cap L\cap F^{-1}(L)=\{z\in K\cap L: F(z)\in L\}
$$
and
$$
\Gamma=\{z\in K\cap L:P_K(z-F(z))\leq_Lz\},
$$
where \(z=(x,u)\).
Theorem \ref{isotonicity-on-cylinder} implies that $P_K$ is $L$-isotone. Since \eqref{si1}$\implies$\eqref{si2}, $I-F$ is $L$-isotone.
Meanwhile, $x_1^1\geq x_2^1\geq\cdots\geq x_p^1\geq\|u^1\|$ implies that $(x^0,y^0)=(0,0)\leq_L(x^1,y^1)$.
Then, by using Proposition \ref{guohan2}, we have that $\varnothing\ne\Omega\subset\Gamma$, the sequence $\{(x^n,u^n)\}$ is convergent, and its limit $(x^*,u^*)$ is a solution of MiCP$(G,H,C,p,q)$. Moreover, $(x^*,y^*)$ is a lower $L$-bound of $\Omega$ and the $L$-least element of $\Gamma$. \hfill $\Box$
\end{proof}
\section{Numerical Example}
Let $L$ be the monotone extended second order cone, then suppose that $K=\R^2\times C$ where $C=\{(u_1,u_2)\in\R^2: u_1\geq u_2\geq 0\}$. Let $f_1(x,u)= \frac{1}{10}x_1-\frac{1}{20}x_2 +\frac{1}{20}\|u\|+1$ and $f_2(x,u)=\frac{1}{5}x_1-\frac{3}{20}x_2+\frac{1}{20}\|u\|-\frac{3}{5}$. Obviously, $f_1(x,u)$ and $f_2(x,u)$ are $L$-monotone. Define $\omega^1:=(2,1, \frac{1}{3},\frac{1}{6})$ and $\omega^2:=(2,1,\frac{1}{6},\frac{1}{3})$; it is easy to find out that\, $\omega^1, \omega^2 \in L$. Then, for two arbitrary vectors $(x,u), (y,v) \in \R^2\times\R^2$ such that $(x,u)\leq_L (y,v)$, by using the definition of the MESOC, we have that $y_1-x_1\geq y_2-x_2\geq\|v-u\|\geq\|u\|-\|v\|$. Hence,
$$
f_1(y,v)-f_1(x,u)=\frac{1}{10}(y_1-x_1)-\frac{1}{20}(y_2-x_2)-\frac{1}{20}(\|u\|-\|v\|)\geq 0,
$$
$$
f_2(y,v)-f_2(x,u)=\frac{1}{5}(y_1-x_1)-\frac{3}{20}(x_2-y_2)-\frac{1}{20}(\|u\|-\|v\|)\geq 0.
$$
Since $\omega^1, \omega^2, (y,v)-(x,u) \in L$, by using the convexity of $L$, if we have $(x,u)\leq_L(y,v)$, then
$$
(f_1(y,v)-f_1(x,u))\omega^1 + (f_2(y,v)-f_2(x,u))\omega^2\in L,
$$
which is equivalent to the following inequality:
$$
f_1(x,u)\omega^1+f_2(x,u)\omega^2\leq_L f_1(y,v)\omega^1+f_2(y,v)\omega^2.
$$
Thus, the mapping $f_1\omega^1+f_2\omega^2$ is $L$-isotone.
Now, we define functions $G$ and $H$ as follows:
$$
G(x,u):=\Bigg(\frac{2}{5}x_1+\frac{2}{5}x_2-\frac{1}{5}\|u\|-\frac{4}{5},-\frac{3}{10}x_1+\frac{6}{5}x_2-\frac{1}{10}\|u\|-\frac{2}{5}\Bigg),
$$
$$
H(x,u):=\Bigg(u_1-\frac{1}{15}x_1+\frac{1}{24}x_2-\frac{1}{40}\|u\|-\frac{7}{30},u_2-\frac{1}{12}x_1+\frac{7}{120}x_2-\frac{1}{40}\|u\|+\frac{1}{30}\Bigg).
$$
Hence, we get
$$
(x-G,u-H)=f_1\omega^1+f_2\omega^2=\Bigg(2f_1+2f_2,f_1+f_2,\frac{1}{3}f_1+\frac{1}{6}f_2,\frac{1}{6}f_1+\frac{1}{3}f_2\Bigg)
$$
is $L-$ isotone. Then we check that all the conditions in Theorem \ref{MainT} are satisfied.
Let us start at the initial condition. We have,
$$
-G(0,0,0,0)=\left(\frac{4}{5},\frac{2}{5}\right)\quad\mbox{\rm and}\quad
\|H(0,0,0,0)\|=\sqrt{\left(-\frac{7}{30}\right)^2+\left(\frac{1}{30}\right)^2}=\frac{\sqrt{2}}{6}.
$$
Evidently, $-G(0,0,0,0)_1\geq -G(0,0,0,0)_2\geq \|H(0,0,0,0)\|.$
Now, consider a vector $(\hat{x},\hat{u}):=(30,12,4,3)\in K$, which yields

$$
G(\hat{x},\hat{u})=\left(15,\frac{9}{2}\right)\quad\mbox{\rm and}\quad
H(\hat{x},\hat{u})=\left(\frac{257}{120},\frac{133}{120}\right).
$$
Moreover, we have that $G(\hat{x},\hat{u})_1\geq G(\hat{x},\hat{u})_2\geq\|H(\hat{x},\hat{u})\|$, which implies that $(\hat{x},\hat{u})\in\Omega$. Hence, $\Omega\neq\varnothing$.
Next, we solve the mixed complementarity problem MiCP$(G,H,C,p,q)$.
For an arbitrary element $(x,y)$, if it is a solution of MiCP$(G,H,C,p,q)$, then
\[
x-G(x,u) = (2f_1+2f_2,f_1+f_2)\text{ where } f_i=f_i(x,u),i=1,2,
\]
and $G(x,u)=0$. Thus, we have $x_1=2f_1+2f_2$ and $x_2=f_1+f_2$.
Moreover,
\begin{equation}\label{value-x1x2}
x_1=\frac{1}{3}\|u\|+\frac{4}{3}\text{ and }x_2=\frac{1}{6}\|u\|+\frac{2}{3}.
\end{equation}
Meanwhile, we have $u\perp H(x,u)$, which implies
\[
\langle u,H(x,u)\rangle=u_1\left(u_1-\frac{1}{3}f_1-\frac{1}{6}f_2\right)+u_2\left(u_2-\frac{1}{6}f_1-\frac{1}{3}f_2\right)=0.
\]
Then,
\begin{equation}\label{norm-u-u1u2}
\|u\|^2=u_1^2+u_2^2=\left(\frac{1}{3}u_1+\frac{1}{6}u_2\right)f_1+\left(\frac{1}{6}u_1+\frac{1}{3}u_2\right)f_2.
\end{equation}
We will figure out all the nonzero solutions on the boundary of $C$. For the first case, without loss of generality, suppose that $u_1=u_2 > 0$, so we have $\|u\|=\sqrt{2}u_1=\sqrt{2}u_2$ and, by using (\ref{norm-u-u1u2}),
\[
u_1=u_2=\frac{1}{4}\left(f_1+f_2\right).
\]
By using the definitions of $f_1$ and $f_2$ as well as  (\ref{value-x1x2}), we get
\begin{equation*}
u_1=u_2=\frac{48+2\sqrt{2}}{287}.
\end{equation*}
Thus, the solution of MiCP$(G,H,C,p,q)$ is
\[
(x,u)=\left(\frac{384+16\sqrt{2}}{287},\frac{192+8\sqrt{2}}{287},\frac{48+2\sqrt{2}}{287},\frac{48+2\sqrt{2}}{287}\right).
\]

For the second case, we consider $u_2=0$, which implies that $\|u\|=u_1$. Hence, the equation (\ref{norm-u-u1u2}) is equivalent to
\[
u_1^2-\left(\frac{1}{3}f_1+\frac{1}{6}f_2\right) u_1=0.
\]
Since $u_1\neq0$, we have
\[
u_1=\frac{1}{3}f_1+\frac{1}{6}f_2.
\]
By using the definitions of $f_1$ and $f_2$, and (\ref{value-x1x2}) again, we have $u_1=\frac{212}{691}$, which implies that $u=\left(\frac{212}{691},0\right)$.
Thus,
\[
(x,u)=\left(\frac{992}{691},\frac{496}{691},\frac{212}{691},0\right).
\]
Consider $(0,0,0,0)$ as a starting point in the fixed point iteration process (\ref{picard iteration 2}). We have
\begin{equation}\label{Picard-applied}
\begin{cases}
\begin{aligned}
x_{n+1}&=x^n-G(x^n,u^n)\\&=(2f_1(x^n,u^n)+2f_2(x^n,u^n),f_1(x^n,u^n)+f_2(x^n,u^n)),\end{aligned}& \\
\begin{aligned}
u_{n+1}&=P_C(u^n-H(x^n,u^n))\\&=P_C\left(\frac{1}{3}f_1(x^n,u^n)+\frac{1}{6}f_2(x^n,u^n),\frac{1}{6}f_1(x^n,u^n)+\frac{1}{3}f_2(x^n,u^n)\right).\end{aligned}&
\end{cases}
\end{equation}
From the above equations we get $x^{n+1}_1\geq x^{n+1}_2$. Moreover, since as the starting point we set $(0,0,0,0)$, then for any arbitrary $i\in\N$, we have that $x_1^i\geq x_2^i\geq 0$.
Define the set $S$ as follows:
\[
S:=\left\{(x,u)\in\R^2\times\R^2:0\leq x_1<\frac{992}{691},0\leq x_2<\frac{496}{691},0\leq u_1<\frac{212}{691},u_2=0 \right\}.
\]
We want to show that for any $n\in\N$ we have $(x^n,u^n)\in S$. We will prove it by induction.
First, we have $(x^0,u^0)\in S$.
Suppose next\, $0\leq x_1^n < \frac{992}{691}$, $0\leq x_2^n < \frac{496}{691}$, $0\leq u_1^n < \frac{212}{691}$ and $u_2=0$, which is equivalent to $\|u^n\|=u_1^n$. Since $x_1^n\geq x_2^n$, we have
\[
\begin{aligned}
0<x_1^{n+1}=2f_1(x^n,u^n)+2f_2(x^n,u^n)
& = \frac{3}{5}x_1^n-\frac{2}{5}x_2^n+\frac{1}{5}u_1^n+\frac{4}{5}\\\\
& < \frac{3}{5}\cdot\frac{992}{691}-\frac{2}{5}\cdot\frac{496}{691}+\frac{1}{5}\cdot\frac{212}{691}+\frac{4}{5}= \frac{992}{691}.
\end{aligned}
\]
Similarly,
\[
\begin{aligned}
0<x_2^{n+1}=f_1(x^n,u^n)+f_2(x^n,u^n)
& = \frac{3}{10}x_1^n-\frac{1}{5}x_2^n+\frac{1}{10}u_1^n+\frac{2}{5}\\\\
& < \frac{3}{10}\cdot\frac{992}{691}-\frac{1}{5}\cdot\frac{496}{691}+\frac{1}{10}\cdot\frac{212}{691}+\frac{2}{5}=\frac{496}{691}.
\end{aligned}
\]
Meanwhile, we also have
\[
u^n-H(x^n,u^n)=\left(\frac{1}{3}f_1(x^n,u^n)+\frac{1}{6}f_2(x^n,u^n),\frac{1}{6}f_1(x^n,u^n)+\frac{1}{3}f_2(x^n,u^n)\right).
\]
Obviously, $(u^n-H(x^n,u^n))_1>0$, then
\[
\begin{array}{rcl}
(u^n-H(x^n,u^n))_1 & < & \displaystyle \frac{1}{3}\left(\frac{1}{10}\cdot\frac{992}{691}-\frac{1}{20}\cdot\frac{496}{691}+\frac{1}{20}\cdot\frac{212}{691}+1\right)\\\\
& & \displaystyle +\frac{1}{6}\left(\frac{1}{5}\cdot\frac{992}{691}-\frac{3}{20}\cdot\frac{496}{691}+\frac{1}{20}\cdot\frac{212}{691}-\frac{3}{5}\right)
=\frac{212}{691}
\end{array}
\]
and $0<(u^n-H(x^n,u^n))_2$.
It is easy to check that the projection of it onto $C$ such that $0\leq u^{n+1}_1 <\frac{691}{212}$ and $u_2^{n+1}=0$, must be given on the ray $\{(u_1,u_2):u_1\geq0,u_2=0\}$. It is equivalent to
\[
u^{n+1}=(u^{n+1}_1,u^{n+1}_2)=P_C\left(u^n-H(x^n,u^n)\right)= \left(\frac{1}{3}f_1(x^n,u^n)+\frac{1}{6}f_2(x^n,u^n),0\right).
\]

\NI Thus, the system of equations (\ref{Picard-applied}) is equivalent to
\begin{equation}\label{iteration1}
\begin{cases}
x_1^{n+1}= \frac{3}{5}x_1^n-\frac{2}{5}x_2^n+\frac{1}{5}u_1^n+\frac{4}{5} , &   \\
x_2^{n+1}= \frac{3}{10}x_1^n-\frac{1}{5}x_2^n+\frac{1}{10}u_1^n+\frac{2}{5}, & \\
u_1^{n+1}= \frac{1}{15}x_1^n-\frac{1}{24}x_2^n+\frac{1}{40}u_1^n+\frac{7}{30}.
\end{cases}
\end{equation}
Moreover, we have $x_1^n=2x_2^n$, so (\ref{iteration1}) is equivalent to
\begin{equation}\label{iteration2}
\begin{cases}
x_1^{n+1}= 2x_2^{n+1}, &   \\
x_2^{n+1}= \frac{2}{5}x_2^n+\frac{1}{10}u_1^n+\frac{2}{5}, & \\
u_1^{n+1}= \frac{11}{120}x_2^n+\frac{1}{40}u_1^n+\frac{7}{30}.
\end{cases}
\end{equation}
The last two lines in (\ref{iteration2}) can be aggregated as follows
$$\left [\begin{array}{c} x_2^{n+1}\\ \mbox{} \vspace{-.3cm} \\ u_1^{n+1} \end{array} \right]=\left [ \begin{array}{rr} \frac{2}{5} & \frac{1}{10}\\ \mbox{} \vspace{-.3cm}\\
\frac{11}{120} & \frac{1}{40} \end{array} \right] \left [\begin{array}{c} x_2^n\\ \mbox{} \vspace{-.3cm}\\ u_1^n \end{array}\right]+
\left [\begin{array}{c} \frac{2}{5} \\ \mbox{} \vspace{-.3cm}\\ \frac{7}{30} \end{array}\right]. $$
One easily verifies that the above $2 \times 2$ matrix has both (real) eigenvalues whose absolute values are less than $1$, so it is a convergent matrix. Hence, the above process is convergent to the unique fixed point $\left[x_2^*~ u_1^*\right]'$ of the above equation, regardless of a starting point $\left[x_2^0~ u_1^0\right]' \in \R^2.$ Explicitly,
$$\left [\begin{array}{c} x_2^*\\ \mbox{} \vspace{-.3cm} \\ u_1^* \end{array} \right]=
\left [ \begin{array}{rr} \frac{3}{5} & -\frac{1}{10}\\ \mbox{} \vspace{-.3cm}\\
-\frac{11}{120} & \frac{39}{40} \end{array} \right]^{-1} \cdot \left [\begin{array}{c} \frac{2}{5} \\ \mbox{} \vspace{-.3cm}\\ \frac{7}{30} \end{array}\right]= \left [\begin{array}{c} \frac{496}{691} \\ \mbox{} \vspace{-.3cm}\\ \frac{212}{691} \end{array}\right]\,.$$
Bearing in mind that $x_1^{n+1}= 2x_2^{n+1}$ and $u_2^n=0$, we have the convergence:
$$(x^n,u^n)=(x_1^n,x_2^n,u_1^n, u_2^n) \rightarrow (x_1^*,x_2^*,u_1^*,0)= \left (\frac{992}{691}, \frac{496}{691}, \frac{212}{691}, 0 \right), $$
which is the same as one solution we have obtained on the boundary.
\begin{remark}
\emph{We remark that $f_1\omega^1+f_2\omega^2$ is not ESOC-isotone.
Indeed, if we assume that $f_1\omega^1+f_2\omega^2$ is $L$-isotone, then for any $(x,u)\leq_L(y,v)$ and $\omega^1,\,\omega^2\in L$, we have
\begin{equation} \label{f_1-f_2}
f_1(x,u)\omega^1+f_2(x,u)\omega^2\leq_L f_1(y,v)\omega^1+f_2(y,v)\omega^2
\end{equation}
and it is equivalent to
$$
(f_1(y,v)-f_1(x,u))\omega^1 + (f_2(y,v)-f_2(x,u))\omega^2\in L.
$$
For arbitrary $(x^*,u^*),(y^*,v^*)\in\R^p\times\R^q$, such that $y^*_1-x^*_1=\|u^*-v^*\|=\|u^*\|-\|v^*\|>0$, $y^*_2-x^*_2=2\|u^*-v^*\|=2(\|u^*\|-\|v^*\|)>0$, it is obvious that $(x^*,u^*)\leq_{\rm ESOC}(y^*,v^*)$. Since $f_1(x,u)= \frac{1}{10}x_1-\frac{1}{20}x_2 +\frac{1}{20}\|u\|+1$ and $f_2(x,u)=\frac{1}{5}x_1-\frac{3}{20}x_2+\frac{1}{20}\|u\|-\frac{3}{5}$,
\[
\begin{aligned}
f_1(y^*,v^*)-f_1(x^*,u^*)&=\frac{1}{10}(y^*_1-x^*_1)-\frac{1}{20}(y^*_2-x^*_2)-\frac{1}{20}(\|u^*\|-\|v^*\|)\\
&=\frac{1}{10}(\|u^*\|-\|v^*\|)-\frac{2}{20}(\|u^*\|-\|v^*\|)-\frac{1}{20}(\|u^*\|-\|v^*\|)\\
&=-\frac{1}{20}(\|u^*\|-\|v^*\|)< 0,
\end{aligned}
\]
\[
\begin{aligned}
f_2(y^*,v^*)-f_2(x^*,u^*)&=\frac{1}{5}(y^*_1-x^*_1)-\frac{3}{20}(x^*_2-y^*_2)-\frac{1}{20}(\|u^*\|-\|v^*\|)\geq 0\\
&=\frac{1}{5}(\|u^*\|-\|v^*\|)-\frac{6}{20}(\|u^*\|-\|v^*\|)-\frac{1}{20}(\|u^*\|-\|v^*\|)\\
&=-\frac{3}{20}(\|u^*\|-\|v^*\|)< 0,
\end{aligned},
\]
contradicting (\ref{f_1-f_2}), so $f_1\omega^1+f_2\omega^2$ is not ESOC-isotone.
Let us recall that both $f_1$ and $f_2$ are MESOC-monotone (which has been proved in the numerical example) and not ESOC-monotone, which implies that $f_1(y,v)-f_1(x,u)$ and $f_2(y,v)-f_2(x,u)$ will not be nonnegative for all $(x,u)\leq_{\rm ESOC}(y,v)$. Since both $f_1(y^*,v^*)-f_1(x^*,u^*)$ and $f_2(y^*,v^*)-f_2(x^*,u^*)$ are negative, then by using convexity of ESOC, since $\omega^1,\,\omega^2\in \textrm{MESOC}\subseteq \textrm{ESOC}$, we have
$$
-(f_1(y^*,v^*)-f_1(x^*,u^*))\omega^1 - (f_2(y^*,v^*)-f_2(x^*,u^*))\omega^2\in \textrm{ESOC}.
$$
Meanwhile, if $f_1\omega_1+f_2\omega_2$ were ESOC-isotone, then
$$
(f_1(y^*,v^*)-f_1(x^*,u^*))w^1 + (f_2(y^*,v^*)-f_2(x^*,u^*))w^2\in \textrm{ESOC}.
$$
Since $\omega^1$, $\omega^2$ are linearly independent, it contradicts pointedness of ESOC.
Thus,
 $f_1\omega^1+f_2\omega^2$ is not ESOC-isotone.}
\end{remark}

\section*{Concluding remarks}

In this paper we study the Monotone Extended Second Order Cone (MESOC) as a new
generalization of the Lorentz cone $\mathcal{L}^n_+$. This cone is different both from $\mathcal{L}^n_+$ and the previously introduced Extended Second Order Cone (ESOC) \cite{nemeth-zhang1,nemeth-zhang2,nemeth-xiao1,nemeth-xiao2,nemeth-xie-zhang,sznajder} in many aspects, but bears similarities too. Both MESOC and ESOC are cones in $\mathbb R^p\times\mathbb R^q$. MESOC is sub-dual as is ESOC, but for $p>1$ it is not self-dual like $\mathcal{L}^n_+$. Both ESOC and MESOC become $\mathcal{L}^{q+1}_+$ when $p=1$.
Contrary to $\mathcal{L}^n_+$, for $q=1$ both ESOC and MESOC are polyhedral. MESOC and ESOC
are symmetric cones for $p=1$ only, that is, if and only if they are Lorentz cones. Contrary to
$\mathcal{L}^n_+$ and ESOC, MESOC is reducible. For both ESOC and MESOC the cylinders $\mathbb R^p\times C$, where $C$ is an arbitrary closed convex set with nonempty interior in $\mathbb R^q$, are isotone projection sets. In fact, these cylinders are isotone projection sets with respect to any intersection of ESOC with $U\times\R^q$, where $U$ is an arbitrary closed convex cone in $\R^p$ (the proof is similar to the first part of the proof of Theorem \ref{isotonicity-on-cylinder}). Contrary to ESOC, any isotone projection set with respect to MESOC is such a cylinder.

We determined the bilinearity rank of MESOC and used the MESOC-isotonicity of the projection onto the cylinder to solve general mixed complementarity problems. We illustrated the corresponding iterative method by using a numerical example with exact numbers. Although the iteration principle for the MESOC is similar to the corresponding one for ESOC,
we remark that there are mixed complementarity problems which can be solved iteratively by MESOC, but the same iterative scheme cannot be used via ESOC, because it does not satisfy the corresponding ESOC-isotonicity condition (merely the MESOC-isotonicity). This is due to the fact that although MESOC is a subset of ESOC, MESOC-isotonicity of mappings does not imply their ESOC-isotonicity. This idea is underlined in the preceding section.


\end{document}